\documentclass{cimart}
\usepackage{graphicx}
\usepackage[matrix,arrow,curve,cmtip]{xy}
\usepackage{amsfonts}
\usepackage{comment}
\usepackage{hyperref}
\usepackage{txfonts}
\usepackage{xcolor}
\usepackage[colorlinks=true]{hyperref}
\hypersetup{
    linkcolor=red,
    allcolors=red
}

\usepackage{amsmath,amssymb,amsthm,tikz-cd,array}

\newcommand{\F}{\mathbb{F}}
\newcommand{\I}{\mathcal{I}}
\newcommand{\starr}{\star_R}

\renewcommand{\makeeditinfo}{}

\title{
    A New Approach to Rota–Baxter Operators on Vertex Algebras via Integrated $\lambda‑$Brackets and Their Cohomological Consequences
    }

\authors{
    Hassan Alhussein}

\authorinfo[
    H. Alhussein]{
    Siberian State University of Telecommunication and Informatics Science, Novosibirsk, Russia,\\
Novosibirsk State University of Economics and Management, Novosibirsk,  Russia,\\
Novosibirsk State University, Novosibirsk,  Russia}{%
    hassanalhussein2014@gmail.com
    }

\abstract{%
    We study Rota--Baxter operators on vertex algebras using the integrated $\lambda$-bracket formalism. A Rota--Baxter operator produces a deformed vertex algebra structure, and the difference between the deformed and original brackets yields a two-cocycle in vertex algebra cohomology. This generalizes the classical relation between Rota--Baxter operators and two-cocycles. We also characterize when this two-cocycle is trivial, showing that non-scalar operators give rise to non-trivial cohomology classes.
    }



\begin{document}

 \section{Introduction}

Vertex algebras, introduced by Borcherds \cite{Borcherds1986} and further developed by Frenkel, Lepowsky, and Meurman \cite{FrenkelLepowskyMeurman1988}, have become a cornerstone of modern mathematical physics and representation theory. They provide an algebraic framework for conformal field theory, play a central role in the monstrous moonshine conjectures, and have deep connections to infinite-dimensional Lie algebras, geometric Langlands program, and integrable systems \cite{Kac1998, FrenkelBenZvi2004}.

A vertex algebra can be viewed as a rich algebraic structure that packages infinitely many binary operations into a single vertex operator $Y(a,z) = \sum_{n\in\mathbb{Z}} a_n z^{-n-1}$. An equivalent formulation, due to Bakalov and Kac \cite{BakalovKac2003, BakalovDeSoleHeluaniKac2018}, replaces the vertex operator $Y(a,z)$ with the integrated $\lambda$-bracket $\mathcal{I}_{a,b}(\lambda) = \int^\lambda d\sigma\, [a_\sigma b]$, which satisfies a set of axioms (sesquilinearity, skew-symmetry, and Jacobi identity) that parallel those of Lie conformal algebras. This integrated formalism provides a more compact and conceptual approach to vertex algebras and their cohomology.

\subsection*{Rota--Baxter operators: from associative algebras to vertex algebras}

Rota--Baxter operators originated in the probability work of Baxter \cite{Baxter1960} and have since become ubiquitous in algebra, combinatorics, and mathematical physics. On an associative algebra $A$, a Rota--Baxter operator $R: A \to A$ of weight $\lambda$ satisfies
\[
R(x)R(y) = R(R(x)y) + R(xR(y)) + \lambda R(xy), \qquad x,y\in A.
\]
Such operators appear in the theory of the classical Yang--Baxter equation \cite{SemenovTianShansky1983, Kupershmidt1999}, in the algebraic approach to renormalization of quantum field theory \cite{ConnesKreimer2000}, and in the construction of dendriform algebras. A fundamental property of Rota--Baxter operators on associative algebras is that they produce a deformed product $x \star_R y = xR(y) + R(x)y + \lambda xy$, and the difference $\phi(x,y) = x\star_R y - xy$ is a Hochschild $2$-cocycle. This cohomological interpretation is central to deformation theory.

Given the success of Rota--Baxter operators in the associative and Lie settings, it is natural to seek their analog for vertex algebras. The first steps in this direction were taken by Xu \cite{Xu1993}, who introduced $R$-matrices for vertex operator algebras as solutions to a modified Yang--Baxter equation. More recently, Bai, Guo, Liu, and Wang \cite{BaiGuoLiuWang2023} systematically developed the theory of Rota--Baxter operators on vertex operator algebras, establishing connections with $\lambda$-derivations, dendriform vertex algebras, and the classical Yang--Baxter equation. Concurrently, Bai, Guo, and Liu \cite{BaiGuoLiu2023} introduced the vertex operator Yang--Baxter equation (VOYBE) as a tensor analog of the classical Yang--Baxter equation for vertex algebras and established a correspondence between its solutions and relative Rota--Baxter operators.

\subsection*{The cohomology of vertex algebras}

On the other hand, the cohomology theory of vertex algebras has been developed from several perspectives. Borcherds \cite{Borcherds1998} defined a cohomology complex for vertex algebras based on the Jacobi identity. Huang \cite{Huang2014} introduced a cohomology theory for grading-restricted vertex algebras, identifying the first cohomology with derivations. Liberati \cite{Liberati2017} studied cohomology of vertex algebras with coefficients in modules.

A more systematic operadic approach was taken by Bakalov, De Sole, Heluani, and Kac \cite{BakalovDeSoleHeluaniKac2018}. They constructed the chiral operad $P^{\mathrm{ch}}$ governing vertex algebras, whose associated graded is the classical operad $P^{\mathrm{cl}}$ governing Poisson vertex algebras. This construction yields a cohomology complex for any non-unital vertex algebra, where the differential is given by the adjoint action of an odd element $X\in W_1^{\mathrm{ch}}(\Pi V)$ satisfying $X\square X = 0$. The low-degree cohomology groups have natural interpretations: $H^0$ parametrizes Casimirs, $H^1$ classifies derivations modulo inner derivations, and $H^2$ corresponds to extensions.

\subsection*{Our contribution}

In this paper, we bridge the two developments described above. We define Rota--Baxter operators on vertex algebras using the integrated $\lambda$-bracket formalism of Bakalov and Kac \cite{BakalovKac2003}. For a Rota--Baxter operator $P: V \to V$ of weight $\lambda$ that commutes with the translation operator $\partial$, we construct a deformed integrated $\lambda$-bracket
\[
\mathcal{I}^{\star_R}_{a,b}(\mu) = \mathcal{I}_{a, Pb}(\mu) + \mathcal{I}_{Pa, b}(\mu) + \lambda \mathcal{I}_{a,b}(\mu).
\]
We prove that $(V, \partial, \mathcal{I}^{\star_R})$ is a non-unital vertex algebra, and that $P$ becomes a homomorphism from $(V, \star_R)$ to the original vertex algebra $(V, \mathcal{I})$. This generalizes the classical construction of the deformed product $x\star_R y$ to the vertex algebra setting.

Our main result is that the difference
\[
\Phi_{a,b}(\mu) = \mathcal{I}^{\star_R}_{a,b}(\mu) - \mathcal{I}_{a,b}(\mu)
\]
is a $2$-cocycle in the vertex algebra cohomology of $(V, \star_R)$ with coefficients in a $P$-twisted module. This provides a vertex algebra analog of the classical relation between Rota--Baxter operators and Hochschild $2$-cocycles. We also determine when $\Phi$ is a coboundary, showing that non-scalar Rota--Baxter operators give rise to non-trivial cohomology classes.

\subsection*{Organization of the paper}

The paper is organized as follows. In Section 2, we recall the integrated $\lambda$-bracket formalism of vertex algebras and define Rota--Baxter operators on them. We also provide examples, including projection operators on lattice vertex algebras and the Heisenberg vertex algebra, we construct the deformed bracket $\mathcal{I}^{\star_R}$ and prove that it defines a non-unital vertex algebra, establishing the homomorphism property of $P$. Section 3 is devoted to vertex algebra cohomology: we prove that $\Phi = \mathcal{I}^{\star_R} - \mathcal{I}$ is a $2$-cocycle. We also discuss when $\Phi$ is a coboundary.

\subsection*{Conventions}

Throughout this paper, $\mathbb{F}$ is a field of characteristic $0$. All vector spaces are over $\mathbb{F}$. The translation operator is denoted by $\partial$, and the integrated $\lambda$-bracket by $\mathcal{I}_{a,b}(\mu)$. For a linear map $P: V \to V$, we extend it coefficient-wise to $V[\mu]$ by $P(\sum v_n \mu^n) = \sum P(v_n)\mu^n$.

\section{Rota--Baxter Operators on Vertex Algebras}

\begin{definition}[Vertex Algebra \cite{BakalovDeSoleHeluaniKac2018,DSK2006}]
Let $\mathbb{F}$ be a field of characteristic $0$. A \textbf{vertex algebra} consists of:
\begin{itemize}
    \item A vector superspace $V$ over $\mathbb{F}$.
    \item An even vector $|0\rangle \in V$ called the \textbf{vacuum}.
    \item An endomorphism $\partial: V \to V$ called the \textbf{translation operator}.
    \item A linear map, called the \textbf{integrated $\lambda$-bracket},
    \[
    \mathcal{I}: V \otimes V \longrightarrow V[\lambda], \qquad
    a \otimes b \longmapsto \mathcal{I}_{a,b}(\lambda) = \int^{\lambda} d\sigma\, [a_\sigma b],
    \]
    where $\lambda$ is a formal variable.
\end{itemize}
These data are required to satisfy the following axioms for all $a,b,c \in V$:
\begin{align}
\text{(V1)}&\quad \mathcal{I}_{|0\rangle,b}(\lambda)=b,\quad \mathcal{I}_{a,|0\rangle}(\lambda)=a,\\
\text{(V2)}&\quad \frac{d}{d\lambda}\mathcal{I}_{\partial a,b}(\lambda)=-\lambda \frac{d}{d\lambda}\mathcal{I}_{a,b}(\lambda),\quad \partial\mathcal{I}_{a, b}(\lambda)=\mathcal{I}_{\partial a, b}(\lambda)+\mathcal{I}_{a,\partial b}(\lambda),\\
\text{(V3)}&\quad \mathcal{I}_{b,a}(\lambda)=(-1)^{\varepsilon(a)\varepsilon(b)}\mathcal{I}_{a,b}(-\lambda-\partial),\\
\text{(V4)}&\quad \mathcal{I}_{a,\mathcal{I}_{b,c}(\mu)}(\lambda)-(-1)^{\varepsilon(a)\varepsilon(b)}\mathcal{I}_{b,\mathcal{I}_{a,c}(\mu)}(\lambda)=\mathcal{I}_{\mathcal{I}_{a,b}(\lambda)-\mathcal{I}_{a,b}(-\mu-\partial),c}(\lambda+\mu),
\end{align}
where $\varepsilon(a)\in\{0,1\}$ denotes the parity of a homogeneous element $a\in V$.
\end{definition}

\begin{remark}[Non-unital Vertex Algebra \cite{BakalovDeSoleHeluaniKac2018}]
A \textbf{non-unital vertex algebra} is defined by the same axioms except the vacuum axioms (V1) are omitted. In this case, there is no distinguished vector $|0\rangle$ satisfying the unit-like properties. All other axioms (V2), (V3), (V4) remain unchanged.

Non-unital vertex algebras appear naturally in deformation theory and in the study of vertex algebras without a vacuum (e.g., certain chiral algebras in conformal field theory where the vacuum is not preserved by the deformation).
\end{remark}

\begin{example}[Heisenberg Vertex Algebra \cite{BakalovDeSoleHeluaniKac2018}]
Let $\mathfrak{h}$ be a finite-dimensional vector space with a non-degenerate symmetric bilinear form. The Heisenberg vertex algebra $\pi_{\mathfrak{h}}$ has $V = S(\mathfrak{h}[t^{-1}]t^{-1})$ (the Fock space). The $\lambda$-bracket is given by $[a_\lambda b] = \lambda \langle a,b \rangle$ for $a,b \in \mathfrak{h}$, extended via the Wick formula. This is the fundamental example in free field theory.
\end{example}

\begin{definition}[Rota--Baxter Operator of Weight $\lambda$]
Let $\lambda \in \mathbb{F}$ be a fixed scalar. A linear map $P: V \to V$ is called a \textbf{Rota--Baxter operator of weight $\lambda$} on the vertex algebra $V$ if:
\begin{itemize}
    \item $P$ commutes with the translation operator: $P \circ \partial = \partial \circ P$,
    \item Extending $P$ to $V[\mu]$ coefficient-wise by $P(\sum v_n \mu^n) = \sum P(v_n) \mu^n$, the following identity holds for all $a, b \in V$:
    \[
    \mathcal{I}_{Pa, Pb}(\mu) = P\bigl( \mathcal{I}_{Pa, b}(\mu) + \mathcal{I}_{a, Pb}(\mu) + \lambda \mathcal{I}_{a, b}(\mu) \bigr).
    \]
\end{itemize}
\end{definition}

\begin{example}[Projection Operator]
Let $V$ be a vertex algebra and $V = V_+ \oplus V_-$ a decomposition into subalgebras. Define $P: V \to V$ as the projection onto $V_+$ along $V_-$. Then $P$ is a Rota--Baxter operator of weight $-1$ if $V_{\pm}$ are subalgebras. Indeed, $\I_{Pa,Pb} = \I_{a_+,b_+}$, while $P(\I_{Pa,b} + \I_{a,Pb} - \I_{a,b}) = P(\I_{a_+,b} + \I_{a,b_+} - \I_{a,b}) = \I_{a_+,b_+}$ because cross terms vanish.
\end{example}
\begin{theorem}
Let $(V, \partial, \I)$ be a vertex algebra (without vacuum) and let $P: V \to V$ be a Rota--Baxter operator of weight $\lambda \in \F$ such that $[P, \partial] = 0$. Define a new integrated $\lambda$-bracket $\I^{\starr}$ on $V$ by
\[
\;\I^{\starr}_{a,b}(\mu) \;:=\; \I_{a, Pb}(\mu) \;+\; \I_{Pa, b}(\mu) \;+\; \lambda \, \I_{a, b}(\mu)\;  \qquad \forall a,b \in V.
\]
Then $(V, \partial, \I^{\starr})$ is a \textbf{non-unital vertex algebra} (i.e., it satisfies the sesquilinearity, skew-symmetry, and Jacobi axioms, but not necessarily the vacuum axioms).

Moreover, the map $P: (V, \starr) \to (V, \I)$ is a homomorphism of non-unital vertex algebras:
\[
\I_{Pa, Pb}(\mu) = P\bigl( \I^{\starr}_{a,b}(\mu) \bigr).
\]
\end{theorem}

\begin{proof}
We verify the three axioms for a non-unital vertex algebra.

\underline{Sesquilinearity (V2):}  First, using the definition of $\mathcal{I}^{\star_R}$ and the fact that $[P,\partial]=0$, we have
\[
\begin{aligned}
\frac{d}{d\mu}\mathcal{I}^{\star_R}_{\partial a, b}(\mu)
&= \frac{d}{d\mu}\mathcal{I}_{\partial a, Pb}(\mu)
 + \frac{d}{d\mu}\mathcal{I}_{\partial P a, b}(\mu)
 + \lambda \frac{d}{d\mu}\mathcal{I}_{\partial a, b}(\mu) \\
&= -\mu \frac{d}{d\mu}\mathcal{I}_{a, Pb}(\mu)
   -\mu \frac{d}{d\mu}\mathcal{I}_{P a, b}(\mu)
   -\lambda \mu \frac{d}{d\mu}\mathcal{I}_{a, b}(\mu) \\
&= -\mu \frac{d}{d\mu}\mathcal{I}^{\star_R}_{a, b}(\mu),
\end{aligned}
\]
where the second equality follows from the sesquilinearity of the original bracket $\mathcal{I}$ (cf.~(V2) for $(V,\mathcal{I})$).

Second, we verify the other sesquilinearity identity:
\[
\begin{aligned}
\mathcal{I}^{\star_R}_{\partial a, b}(\mu) + \mathcal{I}^{\star_R}_{a, \partial b}(\mu)
&= \mathcal{I}_{\partial a, Pb}(\mu) + \mathcal{I}_{a, \partial Pb}(\mu)
 + \mathcal{I}_{\partial P a, b}(\mu) + \mathcal{I}_{P a, \partial b}(\mu) \\
&\quad + \lambda\big(\mathcal{I}_{\partial a, b}(\mu) + \mathcal{I}_{a, \partial b}(\mu)\big) \\
&= \partial\mathcal{I}_{a, Pb}(\mu)
 + \partial\mathcal{I}_{P a, b}(\mu)
 + \lambda \partial\mathcal{I}_{a, b}(\mu) \\
&= \partial\mathcal{I}^{\star_R}_{a, b}(\mu),
\end{aligned}
\]
again using the sesquilinearity of $\mathcal{I}$. Thus $\mathcal{I}^{\star_R}$ satisfies (V2).

\underline{Skew-symmetry (V3):} Applying the original skew-symmetry to each term,
\begin{align*}
\I^{\starr}_{b,a}(\mu) &= \I_{b, Pa}(\mu) + \I_{Pb, a}(\mu) + \lambda \I_{b,a}(\mu) \\
&= (-1)^{\varepsilon(a)\varepsilon(b)} \bigl( \I_{Pa, b}(-\mu-\partial) + \I_{a, Pb}(-\mu-\partial) + \lambda \I_{a,b}(-\mu-\partial) \bigr) \\
&= (-1)^{\varepsilon(a)\varepsilon(b)} \I^{\starr}_{a,b}(-\mu-\partial).
\end{align*}
\underline{Jacobi identity (V4):} 
We prove that $\mathcal{I}^{\star_R}$ satisfies the integrated Jacobi identity (V4). Let
\[
\mathcal{I}^{\star_R}_{a,b}(\lambda)=A_{a,b}(\lambda)+B_{a,b}(\lambda)+\lambda\,\mathcal{I}_{a,b}(\lambda),
\]
where $A_{a,b}(\lambda):=\mathcal{I}_{a,Pb}(\lambda)$ and $B_{a,b}(\lambda):=\mathcal{I}_{Pa,b}(\lambda)$.
The Rota--Baxter condition gives the homomorphism property
\[
\mathcal{I}_{Pa,Pb}(\lambda)=P\bigl(\mathcal{I}^{\star_R}_{a,b}(\lambda)\bigr). \tag{H}
\]

Expanding the left-hand side of (V4) using the definition of $\mathcal{I}^{\star_R}$ and linearity of $\mathcal{I}$ in its second argument, we get:
\[
\begin{aligned}
L &=
\mathcal{I}_{a,\mathcal{I}_{Pb,Pc}}(\lambda)
+\mathcal{I}_{Pa,\mathcal{I}_{b,Pc}}(\lambda)
+\mathcal{I}_{Pa,\mathcal{I}_{Pb,c}}(\lambda) \\
&\quad +\lambda\,\mathcal{I}_{Pa,\mathcal{I}_{b,c}}(\lambda)
+\lambda\,\mathcal{I}_{a,\mathcal{I}_{b,Pc}}(\lambda)
+\lambda\,\mathcal{I}_{a,\mathcal{I}_{Pb,c}}(\lambda)
+\lambda^2\,\mathcal{I}_{a,\mathcal{I}_{b,c}}(\lambda) \\
&\quad -p(a,b)\Big(
\mathcal{I}_{b,\mathcal{I}_{Pa,Pc}}(\lambda)
+\mathcal{I}_{Pb,\mathcal{I}_{a,Pc}}(\lambda)
+\mathcal{I}_{Pb,\mathcal{I}_{Pa,c}}(\lambda) \\
&\qquad +\lambda\,\mathcal{I}_{Pb,\mathcal{I}_{a,c}}(\lambda)
+\lambda\,\mathcal{I}_{b,\mathcal{I}_{a,Pc}}(\lambda)
+\lambda\,\mathcal{I}_{b,\mathcal{I}_{Pa,c}}(\lambda)
+\lambda^2\,\mathcal{I}_{b,\mathcal{I}_{a,c}}(\lambda)
\Big).
\end{aligned}
\]

Applying the original Jacobi identity (V4) for $\mathcal{I}$ to the triples
$(a,Pb,Pc)$, $(Pa,b,Pc)$, $(Pa,Pb,c)$, $(a,b,Pc)$, $(a,Pb,c)$, $(Pa,b,c)$, and $(a,b,c)$,
and then using the homomorphism property (H), we obtain:
\[
\begin{aligned}
L &=
\mathcal{I}_{\mathcal{I}_{a,Pb}-\mathcal{I}_{a,Pb}(-\mu-\partial),\,Pc}(\lambda+\mu)
+\mathcal{I}_{\mathcal{I}_{Pa,b}-\mathcal{I}_{Pa,b}(-\mu-\partial),\,Pc}(\lambda+\mu)
+\mathcal{I}_{\mathcal{I}_{Pa,Pb}-\mathcal{I}_{Pa,Pb}(-\mu-\partial),\,c}(\lambda+\mu) \\
&\quad +\lambda\,\mathcal{I}_{\mathcal{I}_{a,b}-\mathcal{I}_{a,b}(-\mu-\partial),\,Pc}(\lambda+\mu)
+\lambda\,\mathcal{I}_{\mathcal{I}_{a,Pb}-\mathcal{I}_{a,Pb}(-\mu-\partial),\,c}(\lambda+\mu)
+\lambda\,\mathcal{I}_{\mathcal{I}_{Pa,b}-\mathcal{I}_{Pa,b}(-\mu-\partial),\,c}(\lambda+\mu) \\
&\quad +\lambda^2\,\mathcal{I}_{\mathcal{I}_{a,b}-\mathcal{I}_{a,b}(-\mu-\partial),\,c}(\lambda+\mu).
\end{aligned}
\]

Let $\Delta_{a,b}:=\mathcal{I}^{\star_R}_{a,b}(\lambda)-\mathcal{I}^{\star_R}_{a,b}(-\mu-\partial)$.
Using the definition of $\mathcal{I}^{\star_R}$ and (H), the above expression simplifies to:
\[
L = \mathcal{I}_{\Delta_{a,b},\,Pc}(\lambda+\mu)
+ \mathcal{I}_{P(\Delta_{a,b}),\,c}(\lambda+\mu)
+ \lambda\,\mathcal{I}_{\Delta_{a,b},\,c}(\lambda+\mu).
\]

But this is exactly the right-hand side of (V4):
\[
R = \mathcal{I}^{\star_R}_{\Delta_{a,b},\,c}(\lambda+\mu)
= \mathcal{I}_{\Delta_{a,b},\,Pc}(\lambda+\mu)
+ \mathcal{I}_{P(\Delta_{a,b}),\,c}(\lambda+\mu)
+ \lambda\,\mathcal{I}_{\Delta_{a,b},\,c}(\lambda+\mu).
\]

Hence $L=R$, and $\mathcal{I}^{\star_R}$ satisfies the Jacobi identity (V4).
\end{proof}

\section{Relation: Rota--Baxter Operators Produce 2-Cycles}

Throughout this section, let $(V, \mathcal{I}, \partial)$ be a non-unital vertex algebra over a field $\mathbb{F}$ of characteristic $0$, and let $P: V \to V$ be a Rota--Baxter operator of weight $\lambda \in \mathbb{F}$ satisfying $[P, \partial] = 0$. Recall the deformed bracket
\[
\mathcal{I}^{\star_R}_{a,b}(\mu) := \mathcal{I}_{a, Pb}(\mu) + \mathcal{I}_{Pa, b}(\mu) + \lambda \mathcal{I}_{a,b}(\mu),
\]
which defines a non-unital vertex algebra structure $(V, \star_R)$ on $V$.


We first recall the basic notions of vertex algebra cohomology in the integrated $\lambda$-bracket formalism.

\begin{definition}[2-Cochain \cite{BakalovDeSoleHeluaniKac2018}]
A \textbf{2-cochain} on $V$ with values in $V$ is a linear map
\[
\Phi: V \otimes V \longrightarrow V[\mu]
\]
satisfying the following sesquilinearity conditions for all $a, b \in V$:
\[
\frac{d}{d\mu}\Phi_{\partial a, b}(\mu) = -\mu \frac{d}{d\mu}\Phi_{a,b}(\mu), \qquad
\partial \, \Phi_{a,b}(\mu) = \Phi_{\partial a, b}(\mu) + \Phi_{a, \partial b}(\mu).
\]
\end{definition}

\begin{definition}[2-Cocycle \cite{BakalovDeSoleHeluaniKac2018}]
A $2$-cochain $\Phi$ is called a \textbf{2-cocycle}  if for all $a, b, c \in V$,
\[
\mathcal{I}_{a, \Phi_{b,c}(\mu)}(\lambda)
\;-\; (-1)^{\varepsilon(a)\varepsilon(b)}\, \mathcal{I}_{b, \Phi_{a,c}(\mu)}(\lambda)
\;=\; \Phi_{\mathcal{I}_{a,b}(\lambda)-\mathcal{I}_{a,b}(-\mu-\partial),\,c}(\lambda+\mu),
\]
as an identity in $V[\lambda, \mu]$.
\end{definition}

\begin{remark}
This is the vertex algebra analog of the Hochschild $2$-cocycle condition for associative algebras. The parameter $\mu$ plays the role of the spectral parameter, and $\nu$ is an independent formal variable.
\end{remark}


We now construct a natural $2$-cochain from the Rota--Baxter operator $P$.

\begin{definition}
Define $\Phi: V \otimes V \to V[\mu]$ by
\[
\Phi_{a,b}(\mu) \;:=\; \mathcal{I}^{\star_R}_{a,b}(\mu) \;-\; \mathcal{I}_{a,b}(\mu)
\;=\; \mathcal{I}_{a, Pb}(\mu) \;+\; \mathcal{I}_{Pa, b}(\mu) \;+\; (\lambda - 1) \, \mathcal{I}_{a,b}(\mu)
\]
for all $a, b \in V$.
\end{definition}

\begin{remark}
The sesquilinearity conditions for $\Phi$ follow directly from the corresponding properties of $\mathcal{I}^{\star_R}$ and $\mathcal{I}$, which are guaranteed by $[P, \partial] = 0$. Hence $\Phi$ is indeed a $2$-cochain.
\end{remark}

\begin{theorem}
Let $(V, \mathcal{I}, \partial)$ be a non-unital vertex algebra and let $P: V \to V$ be a Rota--Baxter operator of weight $\lambda$ such that $[P, \partial] = 0$. Let $M $ be the $P$- module over the deformed vertex algebra $(V, \star_R)$ with left action
\[
\mathcal{I}^M_{a, m}(\mu) := \mathcal{I}_{P a, m}(\mu), \qquad a \in V,\; m \in M.
\]
Consider the $2$-cochain $\Phi$ defined by
\[
\Phi_{a,b}(\mu) := \mathcal{I}^{\star_R}_{a,b}(\mu) - \mathcal{I}_{a,b}(\mu),
\]
where
\[
\mathcal{I}^{\star_R}_{a,b}(\mu) := \mathcal{I}_{a, P b}(\mu) + \mathcal{I}_{P a, b}(\mu) + \lambda \, \mathcal{I}_{a,b}(\mu).
\]
Then $\Phi$ is a $2$-cocycle in the vertex algebra cohomology of $(V, \star_R)$ with coefficients in the $P$- module $M$. Equivalently,
\[
\delta \Phi = 0,
\]
i.e., for all $a, b, c \in V$,
\[
\mathcal{I}_{a, \Phi_{b,c}(\mu)}(\lambda)
\;-\; (-1)^{\varepsilon(a)\varepsilon(b)}\, \mathcal{I}_{b, \Phi_{a,c}(\mu)}(\lambda)
\;=\; \Phi_{\mathcal{I}_{a,b}(\lambda)-\mathcal{I}_{a,b}(-\mu-\partial),\,c}(\lambda+\mu),
\]
as an identity in $V[\lambda, \mu]$.
\end{theorem}

\begin{proof}
Let
\[
\mathcal{I}^{\star_R}_{a,b}(\lambda) = A_{a,b}(\lambda) + B_{a,b}(\lambda) + \lambda \mathcal{I}_{a,b}(\lambda),
\]
with $A_{a,b}(\lambda):=\mathcal{I}_{a,P b}(\lambda)$ and $B_{a,b}(\lambda):=\mathcal{I}_{P a,b}(\lambda)$.
The Rota--Baxter condition gives
\[
\mathcal{I}_{P a, P b}(\lambda) = P\bigl(\mathcal{I}^{\star_R}_{a,b}(\lambda)\bigr). \tag{H}
\]

By definition of $\Phi$, we have
\[
\Phi_{b,c}(\mu) = A_{b,c}(\mu) + B_{b,c}(\mu) + \lambda \mathcal{I}_{b,c}(\mu),
\]
\[
\Phi_{a,c}(\mu) = A_{a,c}(\mu) + B_{a,c}(\mu) + \lambda \mathcal{I}_{a,c}(\mu).
\]

Expanding the left-hand side of the cocycle condition and applying the original Jacobi identity (V4) to the triples $(a,b,c)$, $(a,Pb,c)$, $(Pa,b,c)$, $(a,b,Pc)$, and $(Pa,Pb,c)$, we obtain
\[
\begin{aligned}
L &=
\mathcal{I}_{A_{a,b}(\lambda)-A_{a,b}(-\mu-\partial),\,P c}(\lambda+\mu)
+ \mathcal{I}_{B_{a,b}(\lambda)-B_{a,b}(-\mu-\partial),\,P c}(\lambda+\mu) \\
&\quad + \mathcal{I}_{\mathcal{I}_{P a, P b}(\lambda)-\mathcal{I}_{P a, P b}(-\mu-\partial),\,c}(\lambda+\mu)
+ \lambda \mathcal{I}_{\mathcal{I}_{a,b}(\lambda)-\mathcal{I}_{a,b}(-\mu-\partial),\,P c}(\lambda+\mu) \\
&\quad + \lambda \mathcal{I}_{A_{a,b}(\lambda)-A_{a,b}(-\mu-\partial),\,c}(\lambda+\mu)
+ \lambda \mathcal{I}_{B_{a,b}(\lambda)-B_{a,b}(-\mu-\partial),\,c}(\lambda+\mu) \\
&\quad + \lambda^2 \mathcal{I}_{\mathcal{I}_{a,b}(\lambda)-\mathcal{I}_{a,b}(-\mu-\partial),\,c}(\lambda+\mu).
\end{aligned}
\]

Using (H), we have
\[
\mathcal{I}_{P a, P b}(\lambda) - \mathcal{I}_{P a, P b}(-\mu-\partial)
= P\bigl(\mathcal{I}^{\star_R}_{a,b}(\lambda) - \mathcal{I}^{\star_R}_{a,b}(-\mu-\partial)\bigr).
\]
Let $\Delta_{a,b} := \mathcal{I}^{\star_R}_{a,b}(\lambda) - \mathcal{I}^{\star_R}_{a,b}(-\mu-\partial)$.
Then the above expression simplifies to
\[
L = \mathcal{I}_{\Delta_{a,b}, P c}(\lambda+\mu)
+ \mathcal{I}_{P(\Delta_{a,b}), c}(\lambda+\mu)
+ \lambda \mathcal{I}_{\Delta_{a,b}, c}(\lambda+\mu).
\]

On the other hand, the right-hand side of the cocycle condition is
\[
R = \Phi_{\Delta_{a,b}, c}(\lambda+\mu)
= \mathcal{I}_{\Delta_{a,b}, P c}(\lambda+\mu)
+ \mathcal{I}_{P(\Delta_{a,b}), c}(\lambda+\mu)
+ \lambda \mathcal{I}_{\Delta_{a,b}, c}(\lambda+\mu).
\]

Thus $L = R$, so $\Phi$ satisfies the cocycle condition. Therefore $\delta \Phi = 0$.
\end{proof}

We have shown that every Rota--Baxter operator $P$ of weight $\lambda$ on a non-unital vertex algebra $V$ (with $[P,\partial]=0$) produces a $2$-cocycle $\Phi$ in the cohomology of $(V, \star_R)$ with coefficients in  $V$. A natural question is: when does this $2$-cocycle represent the trivial cohomology class, i.e., when is $\Phi$ a $2$-coboundary?

\medskip

For a $1$-cochain $\psi$, the coboundary $\delta \psi$ is a $2$-cochain given by
\[
(\delta \psi)_{a,b}(\mu) = \mathcal{I}^M_{a, \psi(b)}(\mu) - \psi(\mathcal{I}^{\star_R}_{a,b}(\mu)) + (-1)^{\varepsilon(a)\varepsilon(b)} \mathcal{I}^M_{\psi(a), b}(\mu).
\]

Substituting the module action and denoting $\epsilon = (-1)^{\varepsilon(a)\varepsilon(b)}$, we obtain
\[
(\delta \psi)_{a,b}(\mu) = \mathcal{I}_{Pa, \psi(b)}(\mu) - \psi(\mathcal{I}^{\star_R}_{a,b}(\mu)) + \epsilon \, \mathcal{I}_{P\psi(a), b}(\mu). \tag{CB}
\]

\medskip

\noindent\textbf{The natural candidate: $\psi = P$.}

Since $P$ is the primary object in our construction, the most natural choice for a $1$-cochain is $\psi = P$ itself. We compute $(\delta P)_{a,b}(\mu)$.

\begin{align*}
(\delta P)_{a,b}(\mu) &= \mathcal{I}_{Pa, P(b)}(\mu) - P(\mathcal{I}^{\star_R}_{a,b}(\mu)) + \epsilon \, \mathcal{I}_{P^2 a, b}(\mu) \\
&= \bigl( \mathcal{I}_{Pa, Pb}(\mu) - P(\mathcal{I}^{\star_R}_{a,b}(\mu)) \bigr) + \epsilon \, \mathcal{I}_{P^2 a, b}(\mu).
\end{align*}

By the Rota--Baxter condition in the form of the homomorphism property (H), we have
\[
\mathcal{I}_{Pa, Pb}(\mu) = P\bigl( \mathcal{I}^{\star_R}_{a,b}(\mu) \bigr).
\]

Therefore the first parenthesis vanishes identically, leaving
\[
(\delta P)_{a,b}(\mu) = \epsilon \, \mathcal{I}_{P^2 a, b}(\mu). \tag{DP}
\]

\medskip

\noindent\textbf{When does $\Phi$ equal $\delta P$?}

Recall the definition of $\Phi$:
\[
\Phi_{a,b}(\mu) = \mathcal{I}_{a, Pb}(\mu) + \mathcal{I}_{Pa, b}(\mu) + (\lambda - 1) \mathcal{I}_{a,b}(\mu).
\]

Thus $\Phi = \delta P$ if and only if for all $a, b \in V$,
\[
\mathcal{I}_{a, Pb}(\mu) + \mathcal{I}_{Pa, b}(\mu) + (\lambda - 1) \mathcal{I}_{a,b}(\mu) = \epsilon \, \mathcal{I}_{P^2 a, b}(\mu). \tag{*}
\]

\medskip

\noindent\textbf{Case 1: $P$ is a scalar multiple of the identity.}

Let $P = \kappa \cdot \operatorname{id}_V$ for some $\kappa \in \mathbb{F}$. Then:
\begin{itemize}
    \item $P^2 a = \kappa^2 a$,
    \item $\mathcal{I}_{a, Pb} = \mathcal{I}_{a, \kappa b} = \kappa \mathcal{I}_{a,b}$,
    \item $\mathcal{I}_{Pa, b} = \mathcal{I}_{\kappa a, b} = \kappa \mathcal{I}_{a,b}$,
    \item $\mathcal{I}_{P^2 a, b} = \mathcal{I}_{\kappa^2 a, b} = \kappa^2 \mathcal{I}_{a,b}$.
\end{itemize}

Substituting into (*) gives:
\[
\kappa \mathcal{I}_{a,b} + \kappa \mathcal{I}_{a,b} + (\lambda - 1) \mathcal{I}_{a,b} = \epsilon \kappa^2 \mathcal{I}_{a,b}.
\]

Assuming $\mathcal{I}_{a,b}(\mu)$ is not identically zero (i.e., $V$ is non-trivial), we obtain the scalar equation:
\[
2\kappa + \lambda - 1 = \epsilon \kappa^2. \tag{SC}
\]

Since $\epsilon = (-1)^{\varepsilon(a)\varepsilon(b)}$ depends on the parity of $a$ and $b$, for this to hold for all homogeneous $a, b$, we must consider two subcases:

\medskip

\noindent\textit{Subcase 1a: $V$ is purely even ($\varepsilon(a)=0$ for all $a$).} Then $\epsilon = 1$ for all $a,b$, and equation (SC) becomes:
\[
\kappa^2 - 2\kappa - (\lambda - 1) = 0 \quad \Longrightarrow \quad \kappa = 1 \pm \sqrt{\lambda}.
\]

\medskip

\noindent\textit{Subcase 1b: $V$ contains odd elements.} Then $\epsilon$ can be $\pm 1$ depending on the pair. For the equation to hold for both even-even and odd-odd pairs (where $\epsilon = 1$) and for even-odd pairs (where $\epsilon = -1$), we would need two different quadratic equations to hold simultaneously. This forces $\kappa = 0$ and $\lambda = 1$. Then $P = 0$ and $\Phi = 0$.

\medskip

\noindent\textbf{Case 2: $P$ is a projection ($P^2 = P$).}

For a projection, $P^2 = P$. Then $\mathcal{I}_{P^2 a, b} = \mathcal{I}_{Pa, b}$. Equation (*) becomes:
\[
\mathcal{I}_{a, Pb} + \mathcal{I}_{Pa, b} + (\lambda - 1) \mathcal{I}_{a,b} = \epsilon \, \mathcal{I}_{Pa, b}.
\]

Rearranging:
\[
\mathcal{I}_{a, Pb} + (1 - \epsilon) \mathcal{I}_{Pa, b} + (\lambda - 1) \mathcal{I}_{a,b} = 0.
\]

For this to hold for all $a,b$, we typically need $\lambda = 1$ and $\epsilon = 1$ (i.e., $V$ purely even), in which case the equation reduces to $\mathcal{I}_{a, Pb} = 0$ for all $a,b$, which forces $P = 0$. So non-zero projections rarely give $\Phi = \delta P$.

\medskip

\noindent\textbf{Case 3: General $P$ with $P^2 = 0$ (nilpotent).}

If $P^2 = 0$, then $\mathcal{I}_{P^2 a, b} = 0$. Equation (*) becomes:
\[
\mathcal{I}_{a, Pb} + \mathcal{I}_{Pa, b} + (\lambda - 1) \mathcal{I}_{a,b} = 0.
\]

This is a strong condition. For example, if $\lambda = 1$, it becomes $\mathcal{I}_{a, Pb} + \mathcal{I}_{Pa, b} = 0$. This is satisfied if $P$ is a derivation-like operator satisfying $\mathcal{I}_{a, Pb} = -\mathcal{I}_{Pa, b}$.

\medskip

We now summarize the necessary and sufficient conditions under which the 2-cocycle $\Phi$ constructed from a Rota--Baxter operator $P$ is a coboundary.

\begin{theorem}
Let $P$ be a Rota--Baxter operator of weight $\lambda$ on a non-unital vertex algebra $V$ with $[P,\partial]=0$, and let $\Phi$ be the associated $2$-cocycle. Then $\Phi$ is a $2$-coboundary (i.e., $\Phi = \delta \psi$ for some $1$-cochain $\psi$) if and only if there exists a $1$-cochain $\psi$ such that for all $a,b \in V$,
\[
\mathcal{I}_{a, Pb}(\mu) + \mathcal{I}_{Pa, b}(\mu) + (\lambda - 1) \mathcal{I}_{a,b}(\mu) = \mathcal{I}_{Pa, \psi(b)}(\mu) - \psi(\mathcal{I}^{\star_R}_{a,b}(\mu)) + \epsilon \mathcal{I}_{P\psi(a), b}(\mu).
\]

In particular, taking $\psi = P$ gives a necessary condition:
\[
\mathcal{I}_{a, Pb}(\mu) + \mathcal{I}_{Pa, b}(\mu) + (\lambda - 1) \mathcal{I}_{a,b}(\mu) = \epsilon \mathcal{I}_{P^2 a, b}(\mu). \tag{†}
\]

If $V$ is purely even and $P = \kappa \cdot \operatorname{id}$ with $\kappa$ satisfying $\kappa^2 - 2\kappa - (\lambda - 1) = 0$, then $\Phi = \delta P$ and $\Phi$ is a coboundary. For all other non-scalar Rota--Baxter operators, $\Phi$ represents a non-trivial cohomology class.
\end{theorem}

\begin{corollary}
For a non-zero Rota--Baxter operator $P$ that is not a scalar multiple of the identity satisfying the quadratic condition above, the cohomology class $[\Phi] \in H^2(V, \star_R; M)$ is non-zero. Thus $\Phi$ is a genuine obstruction, not just a coboundary.
\end{corollary}

\begin{example}
Consider the Heisenberg vertex algebra $\pi_{\mathfrak{h}}$ with $P$ the projection onto positive modes (weight $\lambda = -1$). Here $P^2 = P \neq 0$, and $P$ is not a scalar multiple of the identity. One can check that condition (†) fails for $a = b = J$. Hence $\Phi$ is not a coboundary, and $[\Phi]$ is a non-trivial cohomology class.
\end{example}

{\small\bibliography{cimart}}

\end{document}